\documentclass[letterpaper,11pt]{amsart}
\usepackage[margin=1.2in]{geometry}
\usepackage{amsmath,amsthm,amssymb}
\usepackage{xspace,xcolor}
\usepackage[breaklinks,colorlinks,citecolor=teal,linkcolor=teal,urlcolor=teal,pagebackref,hyperindex]{hyperref}
\usepackage[alphabetic]{amsrefs}
\usepackage[all]{xy}
\usepackage[english]{babel}
\usepackage{enumitem}
\usepackage{tikz,xcolor}
\usepackage{tikz-cd}
\usepackage{mathrsfs}
\usepackage{color}
\usepackage{comment}

\newcounter{intro}

\newtheorem{intro-conjecture}[intro]{Conjecture}
\newtheorem{intro-corollary}[intro]{Corollary}
\newtheorem{intro-theorem}[intro]{Theorem}

\newcommand{\theoremref}[1]{\hyperref[#1]{Theorem~\ref*{#1}}}
\newcommand{\sectionref}[1]{\hyperref[#1]{Section~\ref*{#1}}}
\newcommand{\lemmaref}[1]{\hyperref[#1]{Lemma~\ref*{#1}}}
\newcommand{\definitionref}[1]{\hyperref[#1]{Definition~\ref*{#1}}}
\newcommand{\propositionref}[1]{\hyperref[#1]{Proposition~\ref*{#1}}}
\newcommand{\conjectureref}[1]{\hyperref[#1]{Conjecture~\ref*{#1}}}
\newcommand{\corollaryref}[1]{\hyperref[#1]{Corollary~\ref*{#1}}}
\newcommand{\exampleref}[1]{\hyperref[#1]{Example~\ref*{#1}}}
\newcommand{\remarkref}[1]{\hyperref[#1]{Remark~\ref*{#1}}}
\newcommand{\itemref}[1]{\hyperref[#1]{Item~\ref*{#1}}}
\newcommand{\equationref}[1]{\hyperref[#1]{Equation~(\ref*{#1})}}
\newcommand{\formularef}[1]{\hyperref[#1]{Formula~(\ref*{#1})}}
\newcommand{\conditionref}[1]{\hyperref[#1]{Condition~(\ref*{#1})}}

\theoremstyle{plain}
\newtheorem{thm}{Theorem}[section]

\newtheorem{lem}[thm]{Lemma}
\newtheorem{prop}[thm]{Proposition}
\newtheorem{cor}[thm]{Corollary}

\theoremstyle{definition}
\newtheorem{defi}[thm]{Definition}

\newtheorem{eg}[thm]{Example}

\theoremstyle{remark}
\newtheorem{rmk}[thm]{Remark}

\def\Z{{\mathbf Z}}
\def\Q{{\mathbf Q}}
\def\R{{\mathbf R}}
\def\C{{\mathbf C}}

\def\cD{\mathcal{D}}
\def\cE{\mathcal{E}}

\def\cH{\mathcal{H}}

\def\cK{\mathcal{K}}

\def\cM{\mathcal{M}}

\def\cO{\mathcal{O}}

\def\cT{\mathcal{T}}

\def\.{\cdot}
\def\^{\widehat}

\def\({\left(}
\def\){\right)}

\renewcommand{\and}{ \ \ \text{ and } \ \ }

\DeclareMathOperator{\Gr} {Gr}
\DeclareMathOperator{\DR} {DR}


\title[Characterization and finite descent of local cohomological invariants]{Characterization and finite descent of local cohomological invariants}

\author[B.~Dirks]{Bradley Dirks}

\address{Department of Mathematics, Stony Brook University, Stony Brook, NY 11794-3651, USA}

\email{bradley.dirks@stonybrook.edu}

\author[S.~Olano]{Sebasti\'{a}n Olano}

\address{Department of Mathematics, University of Toronto, 40 St. George St., Toronto, Ontario Canada, M5S 2E4}

\email{seolano@math.toronto.edu}

\author[D.~Raychaudhury]{Debaditya Raychaudhury}

\address{Department of Mathematics, University of Arizona, 617 N. Santa Rita Ave., Tucson, AZ 85721, USA}

\email{draychaudhury@math.arizona.edu}

\begin{document}

\thanks{B.D. was partially supported by the National Science Foundation under grant NSF-MSPRF grant DMS-2303070. D.R. was partially supported by an AMS-Simons travel grant.}

\begin{abstract} 
We provide simple ``left-inverse characterizations'' of the singularity invariants $c(Z)$, $w(Z)$, and ${\rm HRH}(Z)$ of equidimensional varieties $Z$ introduced in \cites{CDOIsolated, CDOR, DOR} respectively. Combining this with a trace morphism, we establish descent results of these invariants for finite surjective morphisms.
\end{abstract}

\maketitle

\section{Introduction} Let $X$ be a complex algebraic variety (meaning separated but not necessarily irreducible). Hodge theory has long been a key player in the study of singularities of $X$, and recently many singularity invariants have been defined using Saito's theory of mixed Hodge modules that extend the classically studied notions of Du Bois and rational singularities. 

In particular, the theory of higher Du Bois and higher rational singularities have been introduced in \cites{MOPW, JKSY}, extended through a ``pre''-variant in \cite{SVV}, and Hodge-theoretic singularity invariants $c(X),w(X)$ and ${\rm HRH}(X)$ emerged through \cites{CDOIsolated, CDOR, DOR}. These latter three invariants are intimately related to higher singularities; for example, normal varieties with pre $k$-Du Bois singularities are pre $k$-rational if and only if their ${\rm HRH}$ level is $\geq k$. Moreover, these three invariants are related by:
$${\rm HRH}(X) = \min\{c(X), w(X)\}.$$

In singularity theory, it is often useful to find a simple characterization of an invariant by means of the existence of left-inverse of appropriate maps. To highlight these, recall that $X$ has Du Bois singularities if and only if the map $\cO_X\to\underline{\Omega}_X^0$ has a left inverse (we provide a proof of this for completeness, see \corollaryref{cor-DuBoisLeftInv} below). The first appearance of this criterion is the work of Kov\'acs \cite{Kovacs}*{Thm. 2.3}. It has since found significant applications in birational geometry (see, for example, \cite{KK}) and inspired many other such results.

A ``left-inverse'' criterion is most often shown by establishing an injectivity theorem. A remarkable fact is that, for Du Bois singularities, there is such an injectivity theorem, first proven in the influential \cite{KS}*{Thm. 3.3} (see another proof \cite{MPLocCoh}*{Cor. B}). Similar injectivity statements for higher singularities have been established in \cites{MPLocCoh, PSV, Kov}.

Our first main result consists of a collection of injectivity statements that yield left-inverse characterizations of the invariants $c(X)$, $w(X)$, and ${\rm HRH}(X)$. Recall that to a variety $X$ one associates the filtered Du Bois complex $(\underline{\Omega}_X^{\bullet},F)$, whose (shifted) graded pieces are denoted by $\underline{\Omega}_X^k$. By a result of Saito \cite{SaitoMHC}, these complexes are obtained, up to shift, by applying the functor $\Gr^F_{-k}\DR$ to the trivial Hodge module $\Q_X^H[\dim X]$. More recently, the \textit{intersection Du Bois complex} $I\underline{\Omega}_X^k$ was introduced in \cites{PSV, PPLefschetz}; it is defined by applying the same functor $\Gr^F_{-k}\DR$ to the intersection complex Hodge module ${\rm IC}_X^H$, and it comes equipped with a natural morphism from $\underline{\Omega}_X^k$. We introduce a third complex, the \textit{perverse Du Bois complex} $P\underline{\Omega}_X^k$, which factors this morphism. It is defined by applying the same functor to $H^0\Q_X^H[\dim X]$, which corresponds to the perverse sheaf $^{\mathfrak{p}}\cH^0(\Q_X[\dim X])$. By definition, the morphism $\underline{\Omega}_X^k \to {\rm I}\underline{\Omega}_X^k$ factors through $P \underline{\Omega}_X^k$. 

\begin{intro-theorem}\label{thm-inj}  Let $X$ be a variety of pure dimension $n$.
\begin{enumerate}
    \item \label{lem-injLem3} Assume $c(X) \geq k-1$. Then the natural morphism
\[  \cE xt^{i}(P\underline{\Omega}_X^k,\omega_X^\bullet) \to \cE xt^{i}(\underline{\Omega}_X^k,\omega_X^\bullet)\]
is injective on cohomology for all $i$.
\item \label{lem-injLem2} Assume $w(X) \geq k-1$. Then the natural morphism
\[ \cE xt^{i}({\rm I}\underline{\Omega}_X^k,\omega_X^\bullet) \to \cE xt^{i}(P\underline{\Omega}_X^k,\omega_X^\bullet)\]
is an isomorphism for $i < k-n$, injective for $i=k-n$ and both the domain and codomain vanish for $i > k-n$.
\end{enumerate}
In particular, if ${\rm HRH}(X)\geq k-1$, then the natural morphism 
\[ \cE xt^{i}({\rm I}\underline{\Omega}_X^k,\omega_X^\bullet) \to \cE xt^{i}(\underline{\Omega}_X^k,\omega_X^\bullet)\]
is an isomorphism for $i < k-n$, injective for $i=k-n$ and both the domain and codomain vanish for $i > k-n$.
\end{intro-theorem}

The last assertion in the above result was proven by S.G. Park in \cite{PSV}*{Thm. 10.5}. Our proof mirrors that of \emph{loc. cit.}, and in fact the result should more appropriately be called a corollary of Park's proof.

As mentioned before, the above leads to the following:

\begin{intro-corollary}\label{cor-leftinv}
Let $X$ be equidimensional.

\begin{enumerate}
\item $c(X) \geq k$ if and only if,  for all $0 \leq p \leq k$, the natural morphism
\[ \underline{\Omega}_X^p \to P\underline{\Omega}_X^p\]
has a left inverse.
\item $w(X) \geq k$ if and only if, for all $0 \leq p \leq k$, the natural morphism
\[ P\underline{\Omega}_X^p \to {\rm I}\underline{\Omega}_X^p\]
has a left inverse.
    \item ${\rm HRH}(X) \geq k$ if and only if, for all $0\leq p  \leq k$, the natural morphism
\[ \underline{\Omega}_X^p \to {\rm I}\underline{\Omega}_X^p\]
has a left inverse.
\end{enumerate}
\end{intro-corollary}

Similar criteria for higher singularities are established in \cites{ADV, KLV}. 

Another phenomenon in singularity theory is that certain classes of singularities descend under finite surjective morphisms. For example, if $Y$ is a $\Q$-homology manifold (meaning $\Q_Y[\dim Y]$ is a semi-simple perverse sheaf) and $X$ is normal, then $X$ is also a $\Q$-homology manifold. This is most clear in the case $Y$ is smooth and $f\colon Y \to X$ is the quotient by the action of a finite group: then $X$ is a so-called ``$V$-manifold'', and known to be a $\Q$-homology manifold. Of course, this example also shows that smoothness doesn't descend.

For Hodge theoretic classes of singularities, like Du Bois or rational singularities, there are other well-known descent results. The established strategy to obtain these results is to first prove a ``left-inverse'' criterion for the singularity class and use an appropriate trace morphism (if it exists) to show that the left-inverses descend. 

We follow the same strategy to obtain descent results of the three singularity invariants. The main difficulty is that the trace morphism must be constructed for more general objects than the ones considered in \cite{Kim}. The following is formal from properties of Saito's theory of mixed Hodge modules:

\begin{intro-theorem} \label{thm-MHMTrace} Let $Y$ be equidimensional and let $f\colon Y \to X$ be a morphism that is
\begin{itemize}
    \item either the quotient by the action of a finite group $G$, or
    \item a finite surjective morphism, with $X$ normal.
\end{itemize}
Then the natural morphism
\[ \Q_X^H \to f_* \Q_Y^H\]
admits a splitting \emph{(here ``Tr'' stands for ``trace'')}
\[ {\rm Tr}_f \colon f_* \Q_Y^H \to \Q_X^H\]
in $D^b({\rm MHM}(X))$.
\end{intro-theorem}

This recovers the trace morphism of \cite{Kim} and \cite{DuBois}*{(5.11.2)}, by applying ${\rm Gr}^F_{-k}{\rm DR}(-)$. We will see that, in fact, the statement holds in any theory of mixed sheaves \cite{SaitoMixedSheaves}, which highlights the formal nature of the argument. We note that this morphism has already essentially been established in \cite{EquivCharClass}*{Lem. 5.3}.

This leads to the following:

\begin{intro-corollary}\label{cor-descent}
Let $Y$ be equidimensional and let $f\colon Y \to X$ be a morphism that is
\begin{itemize}
    \item either the quotient by the action of a finite group $G$, or
    \item a finite surjective morphism, with $X$ normal.
\end{itemize}
Then we have the following inequalities:
\begin{enumerate} 
\item $c(Y)\leq c(X)$.
\item $w(Y)\leq w(X)$.
\item ${\rm HRH}(Y)\leq {\rm HRH}(X)$.
\end{enumerate}
\end{intro-corollary}

\begin{rmk} Our trace map also recovers the inequality ${\rm lcdef}(X) \leq {\rm lcdef}(Y)$, which was shown in \cite{Kim}*{} using the characterization of ${\rm lcdef}(-)$ in terms of the depth of Du Bois complexes \cite{MPLocCoh}*{Cor. 12.6} (see \cite{PSV}*{Cor. 4.3}).

\corollaryref{cor-descent} shows that pre-$k$-rational singularities descend (a result that was proven in \cites{Kim, KLV}) using the fact that pre-$k$-Du Bois singularities descend \cite{SVV}*{Prop. 4.2} and the equivalence
\[ \text{pre-}k\text{ rational} \iff \text{pre-}k\text{ Du Bois and } {\rm HRH}(X) \geq k.\]
\end{rmk}


There is an alternate proof of the above descent result that we now highlight.

Garc\'{i}a L\'{o}pez and Sabbah \cite{HodgeLyubeznik} defined the \emph{Hodge-Lyubeznik numbers} $\lambda_{r,s}^{p,q}(\cO_{X,x})$ given $x\in X$, and consequently their ``intersection-variants'' ${\rm I}\lambda_{r}^{p,q}(\cO_{X,x})$ were considered in \cite{CDOIsolated}. It turns out that these numbers can detect these three invariants. We prove:

\begin{intro-corollary}\label{cor-HL} Let $Y$ be pure dimensional and let $f\colon Y \to X$ be a finite surjective morphism. For any $x\in X$, write $F_x =\{y_1,\dots, y_j\}$. Then
\begin{enumerate}
    \item $\lambda_{r,s}^{p,q}(\cO_{X,x}) \leq \sum_{i=1}^j \lambda_{r,s}^{p,q}(\cO_{Y,y_i})$,
    \item ${\rm I}\lambda_{r}^{p,q}(\cO_{X,x}) \leq \sum_{i=1}^j {\rm I}\lambda_{r}^{p,q}(\cO_{Y,y_i})$.
\end{enumerate}Moreover, we have \small
\[ \dim_{\C} {\rm Gr}^F_p {\rm Gr}^W_{-p-q} \cH^{-r} i_x^* W_{\dim X -1} \cH^0 \Q_X^H[\dim X] \leq \sum_{i=1}^j \dim_{\C} {\rm Gr}^F_p {\rm Gr}^W_{-p-q} \cH^{-r} i_{y_i}^* W_{\dim X -1} \cH^0 \Q_Y^H[\dim Y].\]
\end{intro-corollary}\normalsize

One can also obtain \corollaryref{cor-descent} using \corollaryref{cor-HL}.

Our final result is about the descent of (local analytic) $\Q$-factoriality defect. Given a normal projective variety $X$, one defines its {\it $\Q$-factoriality defect} $$\sigma(X):=\dim_{\Q}\frac{{\rm Div}_{\Q}(X)}{{\rm CDiv}_{\Q}(X)}$$
where ${\rm Div}_{\Q}(X):={\rm Div}(X)\otimes\Q$ and ${\rm CDiv}_{\Q}(X):={\rm CDiv}(X)\otimes\Q$
(here ${\rm Div}(X)$ is the free abelian group of Weil divisors on $X$ and ${\rm CDiv}(X)$ is the subgroup generated by the classes of Cartier divisors). One can also define their local analytic counterparts $\sigma^{\rm an}(X;x)$ given $x\in X$ (see \cites{Kaw, PPFactorial}). We prove the following

\begin{intro-corollary}\label{cor-qfact}
Let $f\colon Y \to X$ be a finite surjective morphism between normal projective varieties. Then 
\begin{enumerate}
    \item $\sigma(X)\leq\sigma(Y)$. 
    \item Let $x\in X$ and consider the analytic germ $(X,x)$. Also consider the analytic germ $(Y,F_x)$ where $F_x = \{y_1,\dots, y_j\}$ is the fiber over $x$, with induced map $f\colon (Y,F_x) \to (X,x)$. If $\sigma^{\rm an}(Y,y_i)$ is finite for all $1\leq i\leq j$, then $\sigma^{\rm an}(X,x)$ is finite, and we have
\[ \sigma^{\rm an}(X,x) \leq \sum_{i=1}^j \sigma^{\rm an}(Y,y_i).\]
\end{enumerate}
\end{intro-corollary}

\noindent{\bf Outline.} In Section \ref{sec-prelim}, we recall the necessary preliminaries. We prove \theoremref{thm-inj} and \corollaryref{cor-leftinv} in Section \ref{sec-inj}. The following Section \ref{sec-trace} outlines the construction of the trace morphism and proves \theoremref{thm-MHMTrace}. The descent results \corollaryref{cor-descent}, \corollaryref{cor-HL} and \corollaryref{cor-qfact} are proven in Section \ref{sec-descent}.

\medskip

\noindent{\bf Acknowledgments.} We thank Hyunsuk Kim, S\'andor Kov\'acs, Lauren\c{t}iu Maxim and Sung Gi Park for helpful conversations.

\section{Preliminaries}\label{sec-prelim} We begin with a quick review of the definition and properties for a theory of mixed sheaves (as defined by Saito \cite{SaitoMixedSheaves}). These theories are essentially an axiomatization of many of the desirable properties of Beilinson's conjectural theory of mixed motivic sheaves. 

Let $A\subseteq \R$ and $k\subseteq \C$ be subfields. A theory of mixed sheaves is the assignment of an $A$-linear abelian category $\cM(X)$ to any $k$-variety $X$ together with a ``forgetful functor'' ${\rm For} \colon \cM(X) \to {\rm Perv}(X^{\rm an}_{\C},A)$ with the following basic properties (and satisfying a partial collection of the six functor formalism, which we will not expand on here):
\begin{itemize} \item ${\rm For}$ is faithful and exact.

\item For any $M\in \cM(X)$, the perverse sheaf ${\rm For}(M)$ is $k$-constructible and quasi-unipotent.

\item Each $M \in \cM(X)$ has a finite increasing ``weight'' filtration $W_\bullet M$ (of subobjects in $\cM(X)$) so that every morphism in $\cM(X)$ is strict with respect to $W_\bullet$.

\item For any $w\in \Z$, the object ${\rm Gr}^W_w M$ is semi-simple.
\end{itemize}

We say an object $M$ is \emph{pure of weight $j$} if ${\rm Gr}^W_w M \neq 0$ implies $w=j$.

\begin{eg} For $k=\C$, one can take $\cM(-) = {\rm MHM}(-,A)$.
\end{eg}

We can consider the derived category $D^b(\cM(X))$ with its induced forgetful functor \[{\rm For} \colon D^b(\cM(X)) \to D^b_{\rm constr}(A_{X^{\rm an}_{\C}})\]
which is no longer faithful, but is \emph{conservative}, meaning an object vanishes if and only if its image under ${\rm For}$ vanishes.

The derived categories $D^b(\cM(X))$ satisfy a full six-functor formalism: we can take duals, tensor products, push-forwards and pull-backs along any morphism of $k$-varieties. Moreover, after analytification, these functors agree with the corresponding functors for perverse sheaves.

\begin{eg} \label{eg-constantObject} For $X = {\rm Spec}(k)$, we have the forgetful functor $$\cM(X) \to {\rm Perv}(X^{\rm an}_{\C}) = {\rm Vect}_{\rm f.d.}(A)$$ (analogous to mapping a mixed Hodge structure to its underlying vector space). The one-dimensional vector space has a canonical lift $A^\cM \in \cM({\rm Spec}(k))$ which is pure of weight $0$. 

Let $\kappa \colon X \to {\rm Spec}(k)$ be the constant map. Then $\kappa^* A^\cM = A^{\cM}_X$ is the ``constant object'' in $D^b(\cM(X))$. It does not lie in $\cM(X)$, and it is not necessarily pure.

By \cite{SaitoMixedSheaves}*{(3.6), Thm. 3.8} (see also \cite{SaitoMHM}*{(4.4.2)} if $\cM(-) = {\rm MHM}(-,A)$), the constant object is uniquely determined by the fact that ${\rm For}(A^\cM_X) = A_{X^{\rm an}_{\C}}$ and that there is a morphism $A^\cM \to \cH^0\kappa_* A^{\cM}_X$ lying over (via ${\rm For}$) the canonical morphism $A \to H^0(X^{\rm an}_{\C},A)$.
\end{eg}

By general properties of six functors in the mixed sheaf theory, we have
\[ A_X^\cM \in D^{\leq \dim X}(\cM(X)),\]
and moreover, for $X$ equidimensional we set
\[ {\rm Gr}^W_{\dim X} \cH^{\dim X} A^{\cM}_X = {\rm IC}_X^{\cM},\]
denoted in this way because it maps via ${\rm For}$ to the Goresky-MacPherson intersection complex ${\rm IC}_{X^{\rm an}_{\C}}$.

We now specialize to the case $k =\C, A = \Q$ and $\cM(-) = {\rm MHM}(-,\Q)$, and so in our notation above, we replace $\cM$ superscripts by $H$.

On a smooth variety $X$, the data of an object $M \in {\rm MHM}(X)$ consists of a tuple 
\[ ((\cM,F,W),(\cK,W),\alpha)\] where $(\cM,F,W)$ is a bi-filtered right $\cD_X$-module, $(\cK,W)$ is a finitely filtered $\Q$-perverse sheaf on $X^{\rm an}$, and $\alpha$ is a filtered isomorphism
\[ \alpha \colon {\rm DR}_X(\cM,W) \cong (\cK,W)\otimes_\Q \C,\]
however not every such tuple gives a mixed Hodge module. Using local embeddings of a singular variety into smooth varieties, it is possible to define the abelian cateogry ${\rm MHM}(X)$ when $X$ is singular.

Recall that the de Rham functor is defined by
\[ {\rm DR}(\cM)= \left[ \cM \otimes \bigwedge^{\dim X} \cT_X \xrightarrow[]{\nabla} \cM \otimes_{\cO_X} \bigwedge^{\dim X -1} \cT_X  \xrightarrow[]{\nabla} \dots \xrightarrow[]{\nabla} \cM \otimes_{\cO_X} \cT_X \xrightarrow[]{\nabla} \cM \right],\]
where the term on the right is placed in cohomological degree $0$ and where $\nabla$ is the connection given by the $\cD_X$-module structure.

If $(\cM,F)$ is a filtered $\cD_X$-module (meaning $F_k \cM$ is a coherent sub-$\cO_X$-module for all $k$ and $F_j \cM \cdot F_{\ell} \cD_X \subseteq F_{k+\ell}\cM$ for all $j,\ell \in \Z$), then we define a filtration on the de Rham complex by\small
\[ F_k {\rm DR}(\cM) =\] \[\left[F_{k-\dim X} \cM \otimes \bigwedge^{\dim X} \cT_X \xrightarrow[]{\nabla} F_{k-\dim X+1}\cM \otimes_{\cO_X} \bigwedge^{\dim X -1} \cT_X  \xrightarrow[]{\nabla} \dots \xrightarrow[]{\nabla} F_{k-1} \cM \otimes_{\cO_X} \cT_X \xrightarrow[]{\nabla} F_k \cM\right],\]\normalsize
where the differentials are not necessarily $\cO$-linear. However, if we consider ${\rm Gr}^F_k {\rm DR}_X(\cM)$, the morphisms become $\cO$-linear and so we have a bounded complex of coherent $\cO$-modules. This gives a functor ${\rm Gr}^F_k {\rm DR}_X \colon {\rm MHM}(X) \to D^b_{\rm coh}(\cO_X)$ which can then be extended to $D^b({\rm MHM}(X))$.

In fact, even if $X$ is singular, for any $p\in \Z$, there are functors \[{\rm Gr}^F_p {\rm DR}_X(-) \colon D^b({\rm MHM}(X)) \to D^b_{\rm coh}(\cO_X).\]

\begin{eg} \label{eg-DuBois} As in \exampleref{eg-constantObject}, there is a constant mixed Hodge module $\Q_X^H \in D^b({\rm MHM}(X))$ for any complex variety $X$. If $X$ is smooth, then $\Q_X^H[\dim X] \in {\rm MHM}(X)$ has underlying filtered $\cD$-module $(\omega_X,F_\bullet)$ with ${\rm Gr}^F_{-\dim X}(\omega_X) = \omega_X$. Almost from the definition, we have
\[ {\rm Gr}^F_{-p}{\rm DR}_X(\Q_X^H[\dim X]) \cong \Omega_X^p[\dim X -p].\]

If $X$ is singular, this equality no longer holds. However, by \cite{SaitoMHC}*{} we have an isomorphism
\[ {\rm Gr}^F_{-p} {\rm DR}(\Q_X^H[\dim X]) \cong \underline{\Omega}_X^p[\dim X-p]\]
for all $0 \leq p \leq \dim X$.
\end{eg}

Assume now $X$ is equidimensional. Consider the exact triangle
\[ \cK_X^{\bullet}\to \Q_X^{H}[\dim X] \to {\rm IC}_X^H \xrightarrow[]{+1}\]
and the short exact sequence
\[ 0 \to W_{\dim X-1} \cH^{\dim X} \Q_X^H \to \cH^{\dim X} \Q_X^H \to {\rm IC}_X^H \to 0.\]

Here $\cK_X^\bullet$ is the \emph{RHM-defect object} (carefully studied in \cite{PPLefschetz}), so-called because it is the obstruction to the variety $X$ being a rational homology manifold (RHM).

\begin{eg} In analogy with \exampleref{eg-DuBois}, we define the \emph{intersection Du Bois complexes}
\[ {\rm I}\underline{\Omega}_X^p = {\rm Gr}^F_{-p} {\rm DR}_X({\rm IC}_X^H)[p-\dim X].\]

By functoriality of ${\rm Gr}^F_{-p}{\rm DR}(-)$, we have a canonical morphism for all $0\leq p\leq \dim X$:
\begin{equation}\label{QtoIC} \underline{\Omega}_X^p \to {\rm I}\underline{\Omega}_X^p.\end{equation}
\end{eg}

Note that the map (\ref{QtoIC}) is induced by the map $\Q_X^H[\dim X] \to {\rm IC}_X^H$ which factors through $H^0\Q_X^H[\dim X]$. Again, by functoriality of ${\rm Gr}^F_{-p}{\rm DR}(-)$, the morphism (\ref{QtoIC}) factors through $\Gr^F_{-p}\DR$ of this object.

\begin{defi}
    Let $X$ be a complex variety. We define the \textit{perverse Du Bois complexes} 
    \[
    P\underline{\Omega}^p_X = \Gr^F_{-p}\DR(H^0(\Q_X^H[\dim X]))[p-\dim X].
    \]
\end{defi}

We recall the following recently defined invariants \cites{DOR,CDOIsolated, CDOR}:
\begin{defi}\label{def} For $X$ a equidimensional variety, define
\[ {\rm HRH}(X) = \sup \{ k \mid {\rm Gr}^F_{-p} {\rm DR}(\cK_X^\bullet) = 0 \text{ for all } p \leq k\},\]
\[ c(X) = \sup \{k \mid {\rm Gr}^F_{-p} {\rm DR}(\tau^{<0}\Q_X^H[\dim X]) = 0 \text{ for all } p \leq k\},\]
\[ w(X) = \sup \{k \mid {\rm Gr}^F_{-p} {\rm DR}(\cH^0 \cK_X^\bullet) =0 \text{ for all } p \leq k\},\]
so that (because $\tau^{<0} \Q_X^H[\dim X] = \tau^{<0} \cK_X^\bullet$) we have
\begin{equation} \label{eq-minHRH} {\rm HRH}(X) = \min\{c(X), w(x)\}.\end{equation}
\end{defi}

\begin{rmk} The variety $X$ is a $\Q$-homology manifold if and only if $\Q_X[\dim X] \to {\rm IC}_X$ is an isomorphism if and only if ${\rm HRH}(X) = +\infty$. The rational homology manifold condition is related to Poincar\'{e} duality, and the invariant ${\rm HRH}(X)$ is related to Poincar\'{e} duality holding up to a certain Hodge filtered piece \cites{PPLefschetz, DOR}.

The variety $X$ is a cohomologically complete intersection if and only if $\Q_X[\dim X]$ is a perverse sheaf if and only if $c(X) = +\infty$ if and only if ${\rm lcdef}(X) = 0$.

The condition $w(X) \geq 0$ is related to the property of weak rationality (see \eqref{itm-weaklyRational} below).
\end{rmk}

Recently, Garc\'{i}a L\'{o}pez and Sabbah \cite{HodgeLyubeznik} defined the \emph{Hodge-Lyubeznik numbers} as a refinement of the usual Lyubeznik numbers, taking into account mixed Hodge theory. We recall the definition here.

\begin{defi} Let $x\in X$. Define the Hodge-Lyubeznik numbers
\[
\lambda_{r,s}^{p,q}(\cO_{X,x}):=\dim_{\C}{\rm Gr}_{p}^F{\rm Gr}_{-p-q}^W\cH^{-r} i_x^*(\cH^{s-\dim X}\Q^H_X[\dim X])
\]
and the intersection Hodge-Lyubeznik numbers
\[
{\rm I}\lambda_{r}^{p,q}(\cO_{X,x}):=\dim_{\C}{\rm Gr}_{p}^F{\rm Gr}_{-p-q}^W\cH^{-r}i_x^*{\rm IC}_X^H.
\]
\end{defi}

The Hodge-Lyubeznik numbers refine the usual Lyubeznik numbers as follows:
\[ \lambda_{r,s}(\cO_{X,x}) = \sum_{p,q} \lambda_{r,s}^{p,q}(\cO_{X,x}).\]

We collect here re-interpretations of the singularity invariants just described.

\begin{thm}[\cite{CDOIsolated}*{Thm. A}, \cite{CDOR}*{Sect. 2.1}] \label{lem-otherDefs} The following statements hold:
\begin{enumerate}
\setlength\itemsep{1em}
    \item The invariant $c(X)$ has the following equivalent characterizations:
    \vspace{0.5em}
    \begin{itemize}
    \setlength\itemsep{0.5em}
         \item $\sup \{k \mid \lambda_{r,s}^{p,q}(\cO_{X,x}) = 0 \text{ for all } x\in X, s < \dim X, q\in \Z, p\geq -k\}$.
    \item $\sup\{k \mid {\rm depth}(\underline{\Omega}_X^p) \geq \dim X -p \text{ for all } p \leq k\}$.
    \item For any codimension $c$ embedding $X\subseteq Y$ into a smooth variety $Y$,
    \[ \sup\{k \mid F_k \cH^{j+c}_X(\cO_Y) = 0 \text{ for all } j>0\}.\]
    \item \label{itm-DuBoisType}
    $ \sup\left\{ k\, \middle|\,  \underline{\Omega}_X^p \to P\underline{\Omega}_X^p \text{ is an isomorphism for all } p\leq k\right\}$.
\end{itemize}
\item The invariant $w(X)$ has the following equivalent characterizations:
\vspace{0.5em}
\begin{itemize} 
\setlength\itemsep{0.5em}
\item  $\sup\left\{k\, \middle|\, \begin{aligned} & {\rm Gr}^F_{p} {\rm Gr}^W_{-p-q} \cH^{-r} i_x^* W_{\dim X -1} \cH^{\dim X} \Q_X = 0\\ & \text{ for all } x\in X, q,r\in \Z, p\geq -k\end{aligned}\right\}$.

\item \label{itm-weaklyRational} For any codimension $c$ embedding $X\subseteq Y$ into a smooth variety $Y$,
\[ \sup\{k \mid F_kW_{n+c} \cH^c_X(\cO_Y) = F_k \cH^c_X(\cO_Y)\}.\]

\item \label{itm-DuBoisTypew} $\sup\left\{ k\, \middle|\,  P\underline{\Omega}_X^p \to {\rm I}\underline{\Omega}_X^p  \text{ is an isomorphism for all } p\leq k\right\}$.
\end{itemize}  
\end{enumerate}
\end{thm}

Similar characterizations of ${\rm HRH}(X)$ hold using the equality \eqref{eq-minHRH}.

The following alternative characterization of $w(X)$ will also prove useful.

\begin{rmk}[\cite{CDOR}] It is not hard to see that if $w(X) \geq k$, then we have equality
\[\lambda^{p,q}_{r,\dim X}(\cO_{X,x})={\rm I}\lambda^{p,q}_r(\cO_{X,x})\textrm{ for all }x\in X,\, q,r\in\Z,\,p\geq -k.\]
\end{rmk}

\section{Injectivity Theorems and Left-Inverse Characterizations}\label{sec-inj}
In this section, we work over $\C$.

Kov\'{a}cs-Schwede \cite{KS}*{Thm. 3.3} found the following useful injectivity result (with another proof \cite{MPLocCoh}*{Cor. B} using mixed Hodge module theory):
\begin{thm}[\cite{KS}*{Thm 3.3}] Consider the natural morphism $\cO_X \to \underline{\Omega}_X^0$. For all $j\in \Z$, the induced morphism on dual objects
\[ \cE xt^j_{\cO_X}(\underline{\Omega}_X^0,\omega_X^\bullet) \to \cH^j \omega_X^\bullet\]
is injective.
\end{thm}

Injectivity theorems such as above gives rise to ``left-inverse characterization'' of singularity invariants. 
For example, as a corollary of the above, we obtain the (now standard) left-inverse characterization of Du Bois singularities (originally due to Kov\'{a}cs \cite{Kovacs}*{Thm. 2.3}):
\begin{cor}[\cite{Kovacs}*{Thm. 2.3})] \label{cor-DuBoisLeftInv} The variety $X$ has Du Bois singularities if and only if the natural morphism
\[ \cO_X \to \underline{\Omega}_X^0.\]
admits a left-inverse.
\end{cor}
\begin{proof} The forward implication is clear. For the converse, assume a left-inverse exists. Thus, a right-inverse to the dual morphism $\mathbb D(\underline{\Omega}_X^0) \to \omega_X^\bullet$ exists. In particular, for all $j\in \Z$, the morphism (induced by applying $\cH^j(-)$)
\[ \cE xt^{j}(\underline{\Omega}_X^0,\omega_X^\bullet) \to \cH^j \omega_X^\bullet\]
is surjective. By the previous theorem, it is injective, too, hence an isomorphism. Thus the induced map on dual objects is an isomorphism, hence so is the original morphism.
\end{proof}

We aim to prove such characterizations for the invariants in \definitionref{def}. To this end, recall the following injectivity result concerning ${\rm HRH}(X)$ (\cite{PSV}*{Thm. 10.5}, due to Sung Gi Park):
\begin{thm}[Sung Gi Park] \label{lem-injLem1} Assume $X$ is equidimensional with ${\rm HRH}(X) \geq k-1$. Then the natural morphism
\[ \cE xt^{i}_{\cO_X}({\rm I}\underline{\Omega}_X^k,\omega_X^\bullet) \to \cE xt^{i}_{\cO_X}(\underline{\Omega}_X^k,\omega_X^\bullet)\]
is an isomorphism for $i < k-n$, injective for $i=k-n$ and the domain vanishes for $i > k-n$.
\end{thm}

\theoremref{thm-inj}\eqref{lem-injLem2} provides an analogue of the above where $P\underline{\Omega}_X^k$ takes the role of $\underline{\Omega}_X^k$, whose proof mirrors that of \lemmaref{lem-injLem1}:

\begin{proof}[Proof of \theoremref{thm-inj}\eqref{lem-injLem2}] As the condition is local, we can assume $X$ is embedded into some smooth variety $Y$. 

The condition $w(Z) \geq k-1$ is equivalent to requiring
\[ {\rm IC}_X^H \subseteq \cH^0 \mathbf D_X^H\]
to be an isomorphism at $F_{k-1-\dim X}$, where $\mathbf D_X^H = \mathbf D(\Q_X^H[\dim X])(-\dim X)$ is (essentially) the dual of $\Q_X^H[\dim X]$.

If we apply ${\rm Gr}^F_{k-\dim X}{\rm DR}(-)$ to the inclusion ${\rm IC}_X^H \subseteq \cH^0 \mathbf D_X^H$, then using the self-duality of ${\rm IC}_X^H$ and the commutativity of ${\rm Gr}^F_{k-\dim X}{\rm DR}(-)$ with the dual functor, we see that we get the morphism
\[ \mathbb D( {\rm I}\underline{\Omega}_X^k) \to \mathbb D({\rm Gr}^F_{-k} {\rm DR}(\cH^{\dim X}\Q_X^H)[k-\dim X]).\]

So we only need to understand this morphism. By definition of the ${\rm Gr}_{k-\dim X}{\rm DR}(-)$ functor, it corresponds to
\small
\[ \begin{tikzcd} \dots \ar[r] & {\rm Gr}^F_{k-2-\dim X} {\rm IC}_X^H \otimes \wedge^2 \cT_Y \ar[d,"\cong"] \ar[r] & {\rm Gr}^F_{k-1-\dim X} {\rm IC}_X^H\otimes \cT_Y \ar[d,"\cong"] \ar[r] & {\rm Gr}^F_{k-1-\dim X} {\rm IC}_X^H\ar[d,hook] \\ \dots \ar[r] & {\rm Gr}^F_{k-2-\dim X} \cH^0\mathbf D_X^H \otimes \wedge^2 \cT_Y\ar[r] & {\rm Gr}^F_{k-1-\dim X} \cH^0 \mathbf D_X^H\otimes\cT_Y \ar[r] & {\rm Gr}^F_{k-1-\dim X} \cH^0 \mathbf D_X^H \end{tikzcd}\]\normalsize
and so the claims are obvious.
\end{proof}

We proceed to prove \theoremref{thm-inj}\eqref{lem-injLem3}. The morphism described in the statement is obtained by dualizing the natural morphism
\[ \Q_X^H[\dim X] \to \cH^0(\Q_X^H[\dim X])\]
to get
\[ \mathbf D(\cH^0(\Q_X^H[\dim X])) \to \mathbf D(\Q_X^H[\dim X]),\]
and then applying ${\rm Gr}^F_{p}{\rm DR}(-)$ and commuting $\mathbf D$ across this functor, we have a natural morphism
\[ R\cH om_{\cO_X}({\rm Gr}^F_{-p}{\rm DR}(\cH^0(\Q_X^H[\dim X]),\omega_X^\bullet) \to R\cH om_{\cO_X}(\underline{\Omega}_X^p[\dim X-p],\omega_X^\bullet).\]

\begin{proof}[Proof of \theoremref{thm-inj}\eqref{lem-injLem3}] The question is local, so we can embed $X$ into a smooth variety $Y$. Assume $\mathbf D(\Q_X^H[\dim X])$ is represented by a bounded complex $M^\bullet$ of mixed Hodge modules on $Y$. By construction of the morphism in the lemma statement, it suffices to prove that the natural morphism
\[ {\rm Gr}^F_{k} {\rm DR}(\cH^0 M^\bullet) \to {\rm Gr}^F_k {\rm DR}(M^\bullet)\]
is injective on cohomology.

The assumption $c(X) \geq k-1$ is equivalent to requiring the natural morphism $F_{k-1} \cH^0(M^\bullet) \to F_{k-1} M^\bullet$ to be a quasi-isomorphism (recall that $F_{k-1}(-)$ is an exact functor). In other words, for all $p < k$ and $j>0$, we have ${\rm Gr}^F_p \cH^j M^\bullet = 0$. Thus, we see that for all $j > 0$, the complex
\[ {\rm Gr}^F_k {\rm DR}(\cH^j M^\bullet) \text{ is equal to } {\rm Gr}^F_k \cH^j M^\bullet \text{ placed in degree 0}.\]

We have the spectral sequence
\[ E_{2}^{i,j} = \cH^i {\rm Gr}^F_k {\rm DR}(\cH^j M^\bullet) \implies \cH^{i+j} {\rm Gr}^F_k {\rm DR}(M^\bullet),\]
and so we see that $E_2^{i,j} \neq 0$ implies either $j > 0$ and $i=0$ or $j = 0$ and $i \leq 0$. In particular, for any $r\geq 2$, if we look at
\[ E_r^{i-r,j+(r-1)} \xrightarrow[]{d_r} E_r^{i,j} \xrightarrow[]{d_r} E_r^{i+r,j-(r-1)}\]
for $j=0$ and $i\leq 0$, the left and right objects both vanish.

So we conclude
\[ E_2^{i,0} = E_\infty^{i,0} \hookrightarrow \cH^{i} {\rm Gr}^F_{k} {\rm DR}(M^\bullet),\]
proving the desired injectivity.
\end{proof}

\begin{rmk} The injectivity results in \theoremref{thm-inj} combine to give another proof of Park's \theoremref{lem-injLem1}. 
\end{rmk}

We can now prove the characterizations of $c(-)$, $w(-)$ and ${\rm HRH}(-)$ in terms of left inverses:

\begin{proof}[Proof of \corollaryref{cor-leftinv}]
    The proofs are the same as \corollaryref{cor-DuBoisLeftInv} above (though using induction on $k$). We only provide the proof sketch for (1).

    The proof is by induction on $k$. For the base case $k=0$, note that the injectivity of \theoremref{thm-inj}\eqref{lem-injLem3} holds, and so the claim is clear.

For $k>0$, we know by induction that $c(X) \geq k-1$. So we can apply the result of \theoremref{thm-inj}\eqref{lem-injLem3} to conclude.
\end{proof}

\section{Trace Morphism for Mixed Sheaves}\label{sec-trace}
Let $\cM(-)$ be an arbitrary theory of $A$-mixed sheaves on $k$-varieties, where $k\subseteq \C$ and $A\subseteq \R$ are subfields.

We aim to prove \theoremref{thm-MHMTrace} from the introduction. The most important case is that of a finite group quotient, so we begin with that.

We review the construction of group invariants of mixed Hodge modules (though we repeat the arguments in the mixed sheaf setting), following \cite{SymmPowerHodge} which concerns the case of $\mathfrak S_n$ acting on $M^{\boxtimes n}$ (for $M \in D^b({\rm MHM}(X))$).

In our setting, let $Y$ be a variety and let $G$ be a finite group acting on $Y$ with quotient $f\colon Y \to X$. For any $g\in G$, we have the action map $g\colon Y \to Y$ satisfying $f\circ g = f$. We consider the object $f_* A_Y^\cM \in D^b(\cM(X))$.

\begin{lem} We have a natural ring homomorphism
\[ \gamma \colon \Q[G] \to {\rm End}_{D^b(\cM(X))}(f_* A_Y^{\cM}).\]
\end{lem}
\begin{proof} As $f \circ g= f$ for all $g\in G$, the action comes from the adjunction for $g \colon Y \to Y$: we have a natural morphism
\[ A_Y^\cM \to g_* g^* A_Y^\cM,\]
and so by applying $f_*$ we get a natural endomorphism
\[ \gamma(g) \colon f_* A_Y^\cM \to (f\circ g)_* A_Y^\cM = f_* A_Y^\cM.\]
\end{proof}

Inside $\Q[G]$ we have the \emph{invariant projector} $\sigma_G = \frac{1}{|G|} \sum_{g\in G} g$. This is an idempotent element (as is easily verified). Hence, we get an idempotent operator
\[ \gamma(\sigma_G) \in {\rm End}_{D^b(\cM(X))}(f_* A_Y^{\cM}).\]

The category $D^b(\cM(X))$ is \emph{idempotent complete} (or \emph{Karoubian}), meaning every idempotent splits (uniquely up to unique isomorphism). This is formal from the existence of the bounded $t$-structure: see \cites{BS,LC}.

Recall that an idempotent endomorphism $e \colon M \to M$ splits if there exists some $N$ and morphisms
\[ M \xrightarrow[]{p} N \xrightarrow[]{i} M\]
such that $i\circ p = e$ and $p\circ i = {\rm id}_N$. The uniqueness up to unique isomorphism statement is regarding the object $N$ and the factorization of $e$.

\begin{defi} We define $(f_* A_Y^\cM)^G$, the $G$-invariant part, to be the object which splits the idempotent $\gamma(\sigma_G)$. This is well-defined up to isomorphism in $D^b(\cM(X))$ and we have morphisms
\[ f_* A_Y^{\cM} \to (f_* A_Y^{\cM})^G \to f_* A_Y^{\cM}.\]
\end{defi}

In fact, this splitting gives us the trace morphism in this setting:
\begin{prop} We have a canonical isomorphism
\[ (f_* A_Y^{\cM})^G \cong A_X^{\cM}\]
and moreover we can choose a projection $f_* A_Y^\cM \to (f_* A_Y^\cM)^G = A_X^{\cM}$ splitting the natural morphism $A_X^\cM \to f_* A_Y^\cM$, which gives ${\rm Tr}_f$ in this setting.
\end{prop}
\begin{proof} For the first claim, recall \cite{SaitoMixedSheaves}*{Thm. 3.8} (or \cite{SaitoMHM}*{(4.4.2)} if $\cM(-) = {\rm MHM}(-)$) that $A_X^H$ is canonically characterized by the property that ${\rm rat}(A_X^\cM) = A_X$ and $\cH^0 \kappa_*(A_X^{\cM}) \cong  H^0_{\cM}(X)$ in the category $\cM({\rm Spec}(k))$.

For the last statement, we have by adjunction
\[ {\rm Hom}_{D^b(\cM(X))}(A_X^\cM , f_* A_Y^\cM) \cong {\rm Hom}_{D^b(\cM({\rm Spec}(k))}(A^\cM, \kappa_* A_Y^\cM),\]
where $\kappa \colon Y \to {\rm pt}$ is the constant map. But by general $t$-structure considerations, the right hand side is equal to
\[ {\rm Hom}_{\cM({\rm Spec}(k))}(A^\cM, H^0_{\cM}(Y)) \hookrightarrow {\rm Hom}_{A}(A,H^0(Y,A)) = A,\]
and so the morphism $A_X^{\cM} = (f_* A_Y^{\cM})^{G} \to f_* A_Y^{\cM}$ is unique up to constant multiples. So, up to rescaling by a non-zero constant (both morphsims are non-zero), the last claim holds.
\end{proof}

\begin{proof}[Proof of \theoremref{thm-MHMTrace}] We have already argued the case of a finite group quotient. Now, assume $X$ is normal. We have
\[ Z \to Y^n \to Y \to X\]
so that $Z\to X$ is a Galois group quotient (for which we have a trace morphism) and $Y^n$ is the normalization of $Y$. Then the desired splitting comes from the composition
\[ A_X^{\cM} \to f_* A_Y^{\cM} \to f_* A_{Y^n}^{\cM} \to f_* A_Z^{\cM} \to A_X^{\cM}.\]
This completes the proof.
\end{proof}

\section{Finite Descent of Singularity Invariants}\label{sec-descent}
Throughout this section, $k= \C, A = \Q$ and $\cM(-) = {\rm MHM}(-,\Q)$.

We give two approaches to the main descent results \corollaryref{cor-descent}: the first uses injectivity lemmas and the second uses Hodge-Lyubeznik numbers. Aside from the $w$-invariant in the first approach, there is no need to appeal to the trace morphism at the level of mixed Hodge modules.

\begin{proof}[First Proof of \corollaryref{cor-descent}] (1) Assume $c(Y) \geq k$ and let $f\colon Y \to X$ be a finite surjective morphism with $X$ normal.

We use the condition $c(X) \geq k$ if and only if $\mathbb D(\underline{\Omega}_X^p) \in D^{[-\dim X,-\dim X+p]}_{\rm coh}(\cO_X)$ for all $p \leq k$. Let $f \colon Y \to X$ be as in the theorem statement.

By \cite{Kim}, the natural morphism $\underline{\Omega}_X^p \to Rf_* \underline{\Omega}_Y^p$ has a left inverse, so we have a factorization of the identity morphism
\[ \underline{\Omega}_X^p \to Rf_* \underline{\Omega}_Y^p \to \underline{\Omega}_X^p.\]

Dualizing, we have a factorization of the identity
morphism \[ \mathbb D(\underline{\Omega}_X^p) \to Rf_* \mathbb D(\underline{\Omega}_Y^p) \to \mathbb D(\underline{\Omega}_X^p).\]

If $f$ is finite, then $Rf_* = f_*$ is exact. Let $p\leq k$ and let $j> p$. We want to prove that $\cH^{-\dim X+j} \mathbb D(\underline{\Omega}_X^p) =0$.

Apply $\cH^{-\dim X+j} = \cH^{-\dim Y+j}$ to the splitting, yielding
\[\cH^{-\dim X+j} \mathbf D(\underline{\Omega}_X^p) \to Rf_* \cH^{-\dim Y +j} \mathbf D(\underline{\Omega}_Y^p) \to \cH^{-\dim X+j} \mathbf D(\underline{\Omega}_X^p)\]
which shows that the identity morphism on the outer term factors through $0$, which gives the vanishing. 

(2) The proof is by induction on $k$, to apply the injectivity result \theoremref{lem-injLem2}\eqref{lem-injLem2}.  Note that the result of the injectivity statement holds unconditionally for $k=0$.

By the trace map construction, we have the commutative diagram
\[ \begin{tikzcd} \cH^0 \Q_X^H[\dim X] \ar[r] \ar[d] & f_* \cH^0 \Q_Y^H[\dim Y] \ar[r] \ar[d] &\cH^0 \Q_X^H[\dim X]\ar[d] \\ {\rm IC}_X^H \ar[r]& f_* {\rm IC}_Y^H \ar[r] & {\rm IC}_X^H \end{tikzcd},\]
because we can apply ${\rm Gr}^W_w \cH^{\dim X}(-)$ to the trace map for $\Q_X^H$.

Applying ${\rm Gr}^F_{-p}{\rm DR}(-)$, we get a commutative diagram
\[ \begin{tikzcd} {\rm Gr}^F_{-p}{\rm DR}\cH^0 \Q_X^H[\dim X] \ar[r] \ar[d] & f_* {\rm Gr}^F_{-p}{\rm DR}\cH^0 \Q_Y^H[\dim Y] \ar[r] \ar[d] &{\rm Gr}^F_{-p}{\rm DR}\cH^0 \Q_X^H[\dim X]\ar[d] \\ {\rm I}\underline{\Omega}_X^p \ar[r]& f_* {\rm I}\underline{\Omega}_Y^p \ar[r] & {\rm I}\underline{\Omega}_X^p \end{tikzcd}\]
where the horizontal compositions are the identity. By assumption, the middle vertical arrow has a left inverse, so certainly the left vertical arrow has one too, as desired.

(3) By the previous two cases and \eqref{eq-minHRH}, we see that ${\rm HRH}(X)$ descends. We can also prove this directly in the spirit of the previous result.

The proof is by induction on $k$, to apply \theoremref{lem-injLem1}.  Note that the result of the injectivity lemma holds unconditionally for $k=0$.

We have the commutative diagram
\[ \begin{tikzcd}  \underline{\Omega}_X^p \ar[r] \ar[d] & f_* \underline{\Omega}_Y^p \ar[r] \ar[d] & \underline{\Omega}_X^p \ar[d] \\ {\rm I}\underline{\Omega}_X^p \ar[r]& f_* {\rm I}\underline{\Omega}_Y^p \ar[r] & {\rm I}\underline{\Omega}_X^p \end{tikzcd}\]
and by assumption, if $p\leq k$, the middle arrow has a left inverse, so the left vertical arrow does, too.
\end{proof}

Above we used the mixed Hodge module version of Kim's trace morphism \cite{Kim} for the invariant $w(X)$. This tool also immediately gives a descent statement \corollaryref{cor-HL} for Hodge-Lyubeznik numbers.

Using the fact that $W_\bullet f_*(-) = f_* W_\bullet(-)$ (because $f$ is finite), the trace morphism $${\rm Tr}_f \colon f_* \Q_Y^H[\dim Y] \to \Q_X^H[\dim X]$$ induces a commutative diagram of splittings (by taking $\cH^0(-)$, $W_{\dim X-1} \cH^0(-)$ and ${\rm Gr}^W_{\dim X} \cH^0(-)$):

\[ \begin{tikzcd} W_{\dim X-1}\cH^0 \Q_X^H[\dim X] \ar[r] \ar[d] & f_* W_{\dim X-1} \cH^0 \Q_Y^H[\dim Y] \ar[r] \ar[d] &W_{\dim X-1}\cH^0 \Q_X^H[\dim X]\ar[d] \\\cH^0 \Q_X^H[\dim X] \ar[r] \ar[d] & f_* \cH^0 \Q_Y^H[\dim Y] \ar[r] \ar[d] &\cH^0 \Q_X^H[\dim X]\ar[d] \\ {\rm IC}_X^H \ar[r]& f_* {\rm IC}_Y^H \ar[r] & {\rm IC}_X^H \end{tikzcd},\]
and also splittings for any $j$:
\[ \cH^j \Q_X^H[\dim X] \to f_* \cH^j \Q_Y^H[\dim Y] \to \cH^j\Q_X^H[\dim X].\]

Now, let $x\in X$ and consider the Cartesian diagram

\[ \begin{tikzcd} F_x \ar[d] \ar[r] & Y\ar[d,"f"] \\ \{x\} \ar[r] & X \end{tikzcd}.\]

As $f$ is proper, we can use base-change after applying $i_x^* \cH^{s-\dim X}(-)$ to our splittings to get \footnotesize
\[ \begin{tikzcd} \cH^{s-\dim X} i_x^* W_{\dim X-1}\cH^0 \Q_X^H[\dim X] \ar[r] \ar[d] & f_* \cH^{s-\dim X} i_{F_x}^* W_{\dim X-1} \cH^0 \Q_Y^H[\dim Y] \ar[r] \ar[d] &\cH^{s-\dim X} i_x^* W_{\dim X-1}\cH^0 \Q_X^H[\dim X]\ar[d] \\\cH^{s-\dim X} i_x^* \cH^0 \Q_X^H[\dim X] \ar[r] \ar[d] & f_* \cH^{s-\dim X} i_{F_x}^* \cH^0 \Q_Y^H[\dim Y] \ar[r] \ar[d] &\cH^{s-\dim X} i_x^* \cH^0 \Q_X^H[\dim X]\ar[d] \\\cH^{s-\dim X} i_x^*{\rm IC}_X^H \ar[r]& f_* \cH^{s-\dim X} i_x^* {\rm IC}_Y^H \ar[r] &\cH^{s-\dim X}  i_x^*{\rm IC}_X^H \end{tikzcd},\]\normalsize
where the columns are short exact sequences and the rows are factorizations of the identity.

Applying ${\rm Gr}^F_p {\rm Gr}^W_{-p-q}$ and computing dimensions, we can proceed to the proof.
\begin{proof}[Proof of \corollaryref{cor-HL}] As $F_x$ is disconnected, the value of $i_{F_x}^*$ applied to any object is equal to $\bigoplus_{i=1}^j i_{y_i}^*$ applied to the same object.

The discussion above shows that the mixed Hodge structure corresponding to $x\in X$ is a direct summand of the corresponding mixed Hodge  structure for $\{y_1,\dots, y_j\}$, which proves the desired inequalities.    
\end{proof}

This implies the descent statement for the other invariants:

\begin{proof}[Second Proof of \corollaryref{cor-descent}] The claim for $c(-)$ follows easily from the first inequality of \corollaryref{cor-HL} (if all vanishings hold upstairs on $Y$, then they also hold on $X$). The claim for ${\rm HRH}(-)$ follows from the claim for $w(-)$.

For $w(-)$ the claim follows from the very last inequality of \corollaryref{cor-HL}.
\end{proof}

We now proceed to prove \corollaryref{cor-qfact}. Let us first observe the following:

\begin{cor} Assume $X,Y$ are normal projective varieties and let $f\colon Y \to X$ be a finite surjective map.

If $\sigma(Y)$ is finite, then $\sigma(X)$ is finite.
\end{cor}
\begin{proof} We have the compatible splittings
\[ \begin{tikzcd} \Q_X^H \ar[r] \ar[d] & f_* \Q_Y^H \ar[r] \ar[d] & \Q_X^H \ar[d] \\ {\rm IC}_X^H[-\dim X] \ar[r] & f_* {\rm IC}_Y^H[-\dim Y] \ar[r] & {\rm IC}_X^H[-\dim X]\end{tikzcd}, \]
which gives a commutative diagram (after applying $\cH^1 \kappa_*(-)$):
\[ \begin{tikzcd} H^1(X) \ar[r] \ar[d] & H^1(Y) \ar[r] \ar[d] & H^1(X) \ar[d] \\ {\rm IH}^1(X) \ar[r] & {\rm IH}^1(Y) \ar[r] & {\rm IH}^1(X)\end{tikzcd}. \]

By \cite{PPFactorial}*{Thm. A}, we know that $\sigma(Y)$ is finite if and only if $H^1(Y) = {\rm IH}^1(Y)$. 
Thus $H^1(X) \to {\rm IH}^1(X)$ (left vertical map) is injective. 
But then the commutativity of the diagram above implies that $H^1(X) \to {\rm IH}^1(X)$ (right vertical map) is also surjective, proving equality and hence the claim.
\end{proof}


\begin{proof}[Proof of \corollaryref{cor-qfact}(1)] There is nothing to prove if $\sigma(Y)=\infty$ whence we assume $\sigma(Y)<\infty$. Thus, by the previous corollary, we obtain $\sigma(X)<\infty$. By \cite{PPFactorial}*{Pf. of Thm. A}, under the assumption that $\sigma(X)$ is finite, we have an isomorphism
\[ {\rm Div}_\Q(X)/ {\rm CDiv}_\Q(X) \cong {\rm coker}(\gamma_X \colon V_X \to {\rm IH}^2(X,\Q) \cap {\rm IH}^{1,1}(X)),\]
where $V_X = \ker(H^2(X,\Q) \to H^2(X,\cO_X))$ and the morphism $\gamma_X$ is induced by
\[ H^2(X,\Q) \to {\rm IH}^2(X,\Q).\]

Once again, we have the compatible splittings
\[ \begin{tikzcd} \Q_X^H \ar[r] \ar[d] & f_* \Q_Y^H \ar[r] \ar[d] & \Q_X^H \ar[d] \\ {\rm IC}_X^H[-\dim X] \ar[r] & f_* {\rm IC}_Y^H[-\dim Y] \ar[r] & {\rm IC}_X^H[-\dim X]\end{tikzcd}, \]
inducing
\[ \begin{tikzcd} H^2(X) \ar[r] \ar[d] & H^2(Y) \ar[r] \ar[d] & H^2(X) \ar[d] \\ {\rm IH}^2(X) \ar[r] & {\rm IH}^2(Y) \ar[r] & {\rm IH}^2(X)\end{tikzcd}. \]

The corollary follows from the observation that $V_X = H^2(X,\Q) \cap V_Y$ as subspaces of $H^2(Y)$. Indeed, we have the commutative diagram
\[ \begin{tikzcd} H^2(X,\Q) \ar[r] \ar[d] & H^2(X,\cO_X) \ar[d] \\ H^2(Y,\Q) \ar[r] & H^2(Y,\cO_Y) \end{tikzcd}\]
where the two vertical morphisms are split injections. Hence, if $v$ in the top left corner maps to $0$ in $H^2(Y,\cO_Y)$, then its image under the top vertical morphism must also be zero.

So we see that ${\rm coker}(\gamma_X)$ is a direct summand of ${\rm coker}(\gamma_Y)$, proving the desired inequality.
\end{proof}

Our final result concerns the analytic $\Q$-factoriality defect. This argument only uses the decomposition theorem and not our trace map, but we include it for completeness.

\begin{cor} Let $f\colon Y \to X$ be a finite surjective morphism between normal projective varieties. Let $x\in X$ and consider the analytic germ $(X,x)$. Also consider the analytic germ $(Y,F_x)$ where $F_x = \{y_1,\dots, y_j\}$ is the fiber over $x$, with induced map $f\colon (Y,F_x) \to (X,x)$.

If $\sigma^{\rm an}(Y,y_i)$ is finite for all $1\leq i\leq j$, then $\sigma^{\rm an}(X,x)$ is finite.

\end{cor}
\begin{proof} By \cite{PPFactorial}*{Thm. B}, the analytic defect is finite if and only if $R^1 \mu_* \cO_{\widetilde{X}} = 0$ for some resolution $\mu \colon \widetilde{X} \to X$. We also let $\mu \colon \widetilde{Y} \to Y$ denote a resolution of $Y$.

By \cite{KebekusSchnell}*{Prop. 8.2}, we know that 
\[{\rm Gr}^F_{0} {\rm DR}_Y({\rm IC}_Y^H) = R\pi_* \cO_{\widetilde{Y}}[\dim Y] \]
\[{\rm Gr}^F_{0} {\rm DR}_X({\rm IC}_X^H) = R\pi_* \cO_{\widetilde{X}}[\dim X]\]
and so, because ${\rm IC}_X^H$ is a direct summand of $f_*{\rm IC}_Y^H$, we see that
\[ R\pi_* \cO_{\widetilde{X}} \text{ is a direct summand of } f_* R\pi_* \cO_{\widetilde{Y}}.\]

Thus, we easily see (by exactness of $f_*$) that if $\sigma^{\rm an}(Y,y)$ is finite, then $\sigma^{\rm an}(X,x)$ is finite, too.
\end{proof}

\begin{proof}[Proof of \corollaryref{cor-qfact}(2)] As before, we are safe to assume $\sigma^{\rm an}(Y,y_i)$ is finite for all $1\leq i\leq j$, and consequently $\sigma^{\rm an}(X,x)<\infty$.

By \cite{PPFactorial}*{Thm. B}, we want to study the morphism
\[ {\rm DR}_X({\rm IC}_X) \to {\rm Gr}^F_0 {\rm DR}_X({\rm IC}_X^H).\]

As our splitting holds at the level of Hodge modules, we have a commutative diagram (whose vertical morphisms are split inclusions):
\[ \begin{tikzcd}{\rm DR}_X({\rm IC}_X) \ar[r] \ar[d] & {\rm Gr}^F_0 {\rm DR}_X({\rm IC}_X^H) \ar[d] \\ f_* {\rm DR}_Y({\rm IC}_Y) \ar[r] & f_* {\rm Gr}^F_0 {\rm DR}_Y({\rm IC}_Y^H) \end{tikzcd}.\]

As the vertical maps are split inclusions, they remain so after applying $\cH^{-\dim X +2}(-)_x$ (taking the cohomology and then the stalk at $x\in X$). Hence, we conclude by exactness of $f$ that we have a commutative diagram with split injective vertical morphisms
\[ \begin{tikzcd} (\cH^{-\dim X +2}{\rm DR}_X({\rm IC}_X))_x \ar[r] \ar[d] & (R^2 \pi_* \cO_{\widetilde{X}})_{x} \ar[d] \\ \bigoplus_{i=1}^j (\cH^{-\dim Y+2}{\rm DR}_Y({\rm IC}_Y))_{y_i} \ar[r] & \bigoplus_{i=1}^j (R^2\pi_* \cO_{\widetilde{Y}})_{y_i} \end{tikzcd},\]
which immediately implies the desired inequality.
\end{proof}

\bibliography{bib}

\begin{bibdiv}
\begin{biblist}

\bib{BS}{article}{
      author={Balmer, Paul},
      author={Schlichting, Marco},
       title={Idempotent completion of triangulated categories},
        date={2001},
     journal={J. Algebra},
      volume={236},
      number={2},
       pages={819\ndash 834},
}

\bib{CDOIsolated}{article}{
      author={{Chen}, Qianyu},
      author={{Dirks}, Bradley},
      author={{Olano}, Sebastian},
       title={{Partial Cohomologically Complete Intersections via Hodge
  Theory}},
        date={2025-11},
     journal={arXiv e-prints},
       pages={arXiv:2511.03042},
      eprint={2511.03042},
}

\bib{CDOR}{article}{
      author={{Chen}, Qianyu},
      author={{Dirks}, Bradley},
      author={{Olano}, Sebastian},
      author={{Raychaudhury}, Debaditya},
       title={{Hodge theory of secant varieties}},
        date={2026-02},
     journal={arXiv e-prints},
       pages={arXiv:2602.04038},
      eprint={2602.04038},
}

\bib{EquivCharClass}{article}{
      author={Cappell, Sylvain~E.},
      author={Maxim, Laurentiu~G.},
      author={Schürmann, Jörg},
      author={Shaneson, Julius~L.},
       title={Equivariant characteristic classes of singular complex algebraic
  varieties},
        date={2012},
     journal={Comm. Pure Appl. Math.},
      volume={65},
      number={12},
       pages={1722\ndash 1769},
         url={https://onlinelibrary.wiley.com/doi/abs/10.1002/cpa.21427},
}

\bib{DuBois}{article}{
      author={du~Bois, Philippe},
       title={Complexe de de {Rham} {filtr\'{e}} {d'une} {vari\'{e}t\'{e}}
  {singuli\`{e}re}},
        date={1981},
     journal={Bull. Soc. Math. France},
      volume={109},
       pages={41\ndash 81},
}

\bib{DOR}{article}{
      author={{Dirks}, Bradley},
      author={{Olano}, Sebastian},
      author={{Raychaudhury}, Debaditya},
       title={{A Hodge Theoretic generalization of $\mathbb{Q}$-Homology
  Manifolds}},
        date={2025-01},
     journal={arXiv e-prints},
       pages={arXiv:2501.14065},
      eprint={2501.14065},
}

\bib{HodgeLyubeznik}{article}{
      author={Garc\'ia~L\'opez, Ricardo},
      author={Sabbah, Claude},
       title={Hodge-{L}yubeznik numbers},
        date={2025},
        ISSN={1631-073X,1778-3569},
     journal={C. R. Math. Acad. Sci. Paris},
      volume={363},
       pages={213\ndash 221},
         url={https://doi.org/10.5802/crmath.724},
      review={\MR{4884743}},
}

\bib{JKSY}{article}{
      author={Jung, Seung-Jo},
      author={Kim, In-Kyun},
      author={Saito, Morihiko},
      author={Yoon, Youngho},
       title={Higher {Du} {Bois} singularities of hypersurfaces},
        date={2022},
     journal={Proc. Lond. Math. Soc.},
      volume={125},
      number={3},
       pages={543\ndash 567},
}

\bib{Kaw}{article}{
      author={Kawamata, Yujiro},
       title={Crepant blowing-up of {$3$}-dimensional canonical singularities
  and its application to degenerations of surfaces},
        date={1988},
        ISSN={0003-486X,1939-8980},
     journal={Ann. of Math. (2)},
      volume={127},
      number={1},
       pages={93\ndash 163},
         url={https://doi.org/10.2307/1971417},
      review={\MR{924674}},
}

\bib{Kim}{article}{
      author={{Kim}, Hyunsuk},
       title={{Trace for the Du Bois complex}},
        date={2025-07},
     journal={arXiv e-prints},
       pages={arXiv:2507.07350},
      eprint={2507.07350},
}

\bib{KK}{article}{
      author={Koll\'ar, J\'anos},
      author={Kov\'acs, S\'andor~J.},
       title={Log canonical singularities are {D}u {B}ois},
        date={2010},
        ISSN={0894-0347,1088-6834},
     journal={J. Amer. Math. Soc.},
      volume={23},
      number={3},
       pages={791\ndash 813},
         url={https://doi.org/10.1090/S0894-0347-10-00663-6},
      review={\MR{2629988}},
}

\bib{KLV}{article}{
      author={{Kov{\'a}cs}, S{\'a}ndor},
      author={{Lank}, Pat},
      author={{Venkatesh}, Sridhar},
       title={{Simple criteria for higher rational singularities}},
        date={2025-07},
     journal={arXiv e-prints},
       pages={arXiv:2507.07351},
      eprint={2507.07351},
}

\bib{Kov}{article}{
      author={{Kov{\'a}cs}, S{\'a}ndor},
       title={{Complexes of differential forms and singularities: The
  injectivity theorem}},
        date={2025-05},
     journal={arXiv e-prints},
       pages={arXiv:2505.09912},
      eprint={2505.09912},
}

\bib{Kovacs}{article}{
      author={Kov\'{a}cs, S\'{a}ndor},
       title={Rational, log canonical, {D}u {B}ois singularities: on the
  conjectures of {K}oll\'{a}r and {S}teenbrink},
        date={1999},
     journal={Compos. Math.},
      volume={118},
      number={2},
       pages={123\ndash 133},
}

\bib{KS}{incollection}{
      author={Kov\'acs, S\'andor~J.},
      author={Schwede, Karl},
       title={Du {B}ois singularities deform},
        date={2016},
   booktitle={Minimal models and extremal rays ({K}yoto, 2011)},
      series={Adv. Stud. Pure Math.},
      volume={70},
   publisher={Math. Soc. Japan, [Tokyo]},
       pages={49\ndash 65},
         url={https://doi.org/10.2969/aspm/07010049},
      review={\MR{3617778}},
}

\bib{KebekusSchnell}{article}{
      author={Kebekus, Stefan},
      author={Schnell, Christian},
       title={Extending holomorphic forms from the regular locus of a complex
  space to a resolution of singularities},
        date={2021},
     journal={J. Amer. Math. Soc.},
      volume={34},
      number={2},
       pages={315\ndash 368},
}

\bib{LC}{article}{
      author={Le, Jue},
      author={Chen, Xiao-Wu},
       title={Karoubianness of a triangulated category},
        date={2007},
        ISSN={0021-8693},
     journal={J. Algebra},
      volume={310},
      number={1},
       pages={452\ndash 457},
  url={https://www.sciencedirect.com/science/article/pii/S0021869306008076},
}

\bib{MOPW}{article}{
      author={Musta\c{t}\u{a}, Mircea},
      author={Olano, Sebasti\'{a}n},
      author={Popa, Mihnea},
      author={Witaszek, Jakub},
       title={The {Du} {Bois} complex of a hypersurface and the minimal
  exponent},
        date={2023},
     journal={Duke Math. J.},
      volume={173},
      number={2},
       pages={1411–\ndash 1436},
}

\bib{MPLocCoh}{article}{
      author={Musta\c{t}\u{a}, Mircea},
      author={Popa, Mihnea},
       title={Hodge filtration on local cohomology, {D}u {B}ois complex and
  local cohomological dimension},
        date={2022},
     journal={Forum Math. Pi},
      volume={10},
       pages={Paper No. e22, 58},
  url={https://doi-org.myaccess.library.utoronto.ca/10.1017/fmp.2022.15},
      review={\MR{4491455}},
}

\bib{SymmPowerHodge}{article}{
      author={Maxim, Lauren\c{t}iu},
      author={Saito, Morihiko},
      author={Sch\"{u}rmann, J\"{o}rg},
       title={Symmetric products of mixed {H}odge modules},
        date={2011},
     journal={J. Math. Pures Appl.},
      volume={96},
      number={5},
       pages={462\ndash 483},
}

\bib{PPLefschetz}{article}{
      author={{Park}, Sung~Gi},
      author={{Popa}, Mihnea},
       title={{Hodge symmetry and Lefschetz theorems for singular varieties}},
        date={2025-08},
     journal={arXiv e-prints},
       pages={arXiv:2410.15638},
      eprint={2410.15638},
}

\bib{PPFactorial}{article}{
      author={{Park}, Sung~Gi},
      author={{Popa}, Mihnea},
       title={{Q-factoriality and Hodge-Du Bois theory}},
        date={2025-08},
     journal={arXiv e-prints},
       pages={arXiv:2508.17748},
      eprint={2508.17748},
}

\bib{PSV}{article}{
      author={{Popa}, Mihnea},
      author={{Shen}, Wanchun},
      author={{Vo}, Anh~Duc},
       title={{Injectivity and Vanishing for the Du Bois Complexes of Isolated
  Singularities}},
        date={2024-09},
     journal={arXiv e-prints},
       pages={arXiv:2409.18019},
      eprint={2409.18019},
}

\bib{SaitoMHC}{article}{
      author={Saito, Morihiko},
       title={Mixed {H}odge complexes on algebraic varieties},
        date={2000},
     journal={Math. Ann.},
      volume={316},
      number={2},
       pages={283\ndash 331},
}

\bib{SaitoMHM}{article}{
      author={Saito, Morihiko},
       title={Mixed {H}odge modules},
        date={1990},
        ISSN={0034-5318},
     journal={Publ. Res. Inst. Math. Sci.},
      volume={26},
      number={2},
       pages={221\ndash 333},
         url={https://doi.org/10.2977/prims/1195171082},
      review={\MR{1047415}},
}

\bib{SaitoMixedSheaves}{article}{
      author={Saito, Morihiko},
       title={On the formalisme of mixed sheaves},
        date={1991},
     journal={RIMS Preprint},
      volume={784},
}

\bib{SVV}{article}{
      author={{Shen}, Wanchun},
      author={{Venkatesh}, Sridhar},
      author={{Vo}, Anh~Duc},
       title={{On $k$-Du Bois and $k$-rational singularities}},
        date={2023-06},
     journal={arXiv e-prints},
       pages={arXiv:2306.03977},
      eprint={2306.03977},
        note={to appear in Ann. Inst. Fourier},
}

\bib{ADV}{article}{
      author={Vo, Anh~Duc},
       title={{Du Bois complexes and singularity theory}},
     journal={Ph.D. Dissertation, Harvard University, Cambridge, MA, 2025},
}

\end{biblist}
\end{bibdiv}
\bibliographystyle{abbrv}

\end{document}